\documentclass[10pt]{article}
\usepackage[utf8]{inputenc}

\usepackage{amsfonts}
\usepackage{amsthm}
\usepackage[fleqn]{amsmath}
\usepackage{mathrsfs} 
\usepackage{ amssymb }
\usepackage[margin=1.2in,footskip=0.3in]{geometry}
\usepackage{hyperref}  

\title{A computer algorithm for the BGG resolution}
\author{Nicolas Hemelsoet\thanks{University of Geneva, \href{mailto:nicolas.hemelsoet@gmail.com}{nicolas.hemelsoet@unige.com}}, Rik Voorhaar\thanks{University of Geneva, \href{mailto:wh.voorhaar@gmail.com}{wh.voorhaar@gmail.com}}}
\date{}

\usepackage{cleveref}
\usepackage{graphicx}
\usepackage{tikz}
\usetikzlibrary{matrix,arrows,positioning, decorations.pathmorphing,backgrounds,fit,petri,calc,cd}
\usepackage{tkz-euclide}
\usepackage{float}
\usepackage{ stmaryrd }
\usepackage{mathtools}
\usepackage[boxed]{algorithm2e}

\usepackage[style=alphabetic,citestyle=alphabetic]{biblatex}
\renewbibmacro{in:}{}

\graphicspath{ {imag
\usepackage[all,cmtip]{xy}es/} }
\usepackage[all]{xy}
\usepackage{mathtools}

\newcommand{\C}{\mathbb{C}}

\renewcommand{\O}{\mathcal O}

\newcommand{\SL}{\rm{SL}}
\newcommand{\Hom}{\text{Hom}}

\newcommand{\g}{\mathfrak g}
\newcommand{\h}{\mathfrak h}

\newcommand{\n}{\mathfrak n}
\renewcommand{\u}{\mathfrak u}
\newcommand{\p}{\mathfrak p}

\newcommand{\BGG}{\text{BGG}}
\newcommand{\ad}{\text{ad}}
\renewcommand{\b}{\mathfrak b}
\renewcommand{\sl}{\mathfrak{sl}}

\newcommand{\sym}{\mathrm{Sym}}

\renewcommand{\ad}{\text{ad}}

\renewcommand{\sl}{\mathfrak{sl}}

\newcommand{\ch}{\text{ch}}

\renewcommand{\ch}{\text{ch}}
\newcommand{\wt}{\mathrm{wt}}

\usepackage{enumerate}

\newtheorem{theorem}{Theorem}[section]
\newtheorem{lemma}[theorem]{Lemma}

\newtheorem{remark}[theorem]{Remark}
\newtheorem{proposition}[theorem]{Proposition}

\newtheorem{definition}[theorem]{Definition}

\begin{document}

\maketitle

\begin{abstract}
    We present a computer algorithm to explicitly compute the BGG resolution and its cohomology. We give several applications, in particular computation of various sheaf cohomology groups on flag varieties. An implementation of the algorithm is available at \url{https://github.com/RikVoorhaar/bgg-cohomology}.
\end{abstract}

\tableofcontents

\section{Introduction}

Let $\g$ be a simple complex finite-dimensional Lie algebra and let $\lambda \in P^+$ be a dominant integral weight. Bernstein-Gelfand-Gelfand \cite{BGG} introduced a complex giving a resolution of the simple module $L(\lambda)$ by direct sums of Verma modules on the form $M(w \cdot \lambda)$, where $w \in W$ is an element of the Weyl group and $w \cdot \lambda$ is the `dot action' defined by $w \cdot \lambda = w(\lambda + \rho) - \rho$. Here $\rho$ is half the sum of positive roots. This resolution is known as the BGG resolution, and the $k$-th term of the BGG complex is 
\[
    \bigoplus_{\{w\in W\,\mid\,\ell(w) = k\}}\hspace{-2em} M(w \cdot \lambda).
\]
On the character level, this gives the Weyl character formula: 
\begin{equation}
    \ch(L(\lambda)) = \sum_{w \in W} (-1)^{\ell(w)}\ch M(w \cdot \lambda).
\end{equation}

Several applications of the BGG resolution have been obtained in different domains. For example Gabber-Joseph \cite{GJ} constructed a resolution of certain primitive quotients of $U(\g)$. Certain kind of vanishing theorems and character formulas for $\text{Ext}$-groups in the BGG category $\O$ can be obtained from the resolution, as explained in \cite{humphreys}. One can also study certain differential operators on homogenous spaces (see \cite{penrose}), where morphisms between Verma modules are interpreted as differential operators. It also appeared in geometric representation theory, for example in \cite{schechtman} where complexes similar to BGG complexs are constructed using local systems on the blow-up of configuration spaces. Recently a BGG-type resolution was even constructed for quiver-Hecke algebras and Cherednik algebras in positive characteristic \cite{bowman}. 

Our main motivation was the work of Lachowska and Qi \cite{LQ}, which described the Hochschild cohomology of the small quantum group as the cohomology of certain coherent sheaves on the Springer resolution. By Bott's theorem, such a cohomology can be obtained using equivariant Lie algebra cohomology. Lachowska-Qi created an algorithm involving the BGG resolution to compute the Lie algebra cohomology, and could compute the dimension of the center of the small quantum group for $\sl_3$ by hand and $\sl_4$ with the help of code in Python written by Bryan Ford for this specific case. \

As the rank grows, the algorithm quickly becomes too complicated to compute by hand. In this paper we propose a computer implementation of this algorithm. Let $\h \subset \b \subset \g$ be a Cartan subalgebra contained in a fixed Borel subalgebra of a complex finite-dimensional Lie algebra $\g$. Using the BGG resolution, our computer algorithm computes the equivariant Lie algebra cohomology $H^{\bullet}(\b, \h, E)$, for a given $\b$-module $E$, with its natural $G$-module structure. These groups can be interpreted as the sheaf cohomology groups $H^{\bullet}(G/B, G \times_B E)$. The algorithm can also be used to compute the maps in the BGG resolution, that are of independent interest. The algorithm is implemented in SageMath \cite{sagemath} and the code is available at \url{https://github.com/RikVoorhaar/bgg-cohomology}. In a separate paper we will focus on applications to the small quantum group. \

The outline of this paper is as follows: in section~\ref{sec:BGGTheory} we recall basic facts about the BGG resolution. In section~\ref{sec:Flag} we explain the relationship with the cohomology of flag varieties and work out two examples. In section~\ref{sec:Algorithm} we present our algorithm, and give some relevant details about its implementation. Some potential extensions and improvements to the algorithm are discussed in section~\ref{sec:Extensions}. Finally, in the section~\ref{sec:Examples} we present several computations obtained with our computer algorithm. 

\subsection*{Acknowlegements}

We would like to thank Anna Lachowska for suggesting the project and for many useful discussions. We are also grateful to NCCR Swissmap for providing partial support for this research.

\section{The BGG resolution}\label{sec:BGGTheory}

In this section, we recall important facts about the BGG resolution of $L(\lambda)$ for $\lambda \in P^+$, mainly following chapter $6$ of \cite{humphreys}.

\subsection{Notation and background}

We fix a simple Lie algebra $\g$ over the complex numbers, and a Borel subalgebra $\b \supset \h$ containing a Cartan subalgebra. Let $\n = [\b,\b]$ and $\u = \n^{\vee}$. Our convention is that $\b = \b^-$, so that the set of positive roots corresponds to $\u$. We pick Chevalley generators $e_i, f_i$ and $h_i$ with the usual convention that the $f_i$ span $\n$. Using the Killing form we can identify $\u$ as the subalgebra generated by the $e_i$. As usual, let $P$ be the weight lattice and $P^+$ be the set of integral dominant weights. For $\lambda \in P^+$ the corresponding simple module is written as $L(\lambda)$, and for $\mu \in \h^*$ we write $M(\mu)$ for the corresponding \emph{Verma module} defined as $M(\mu) \coloneqq \text{Ind}_{U(\b)}^{U(\g)} \mathbb C_{\mu}$. Here $\C_{\mu}$ is the one-dimensional $U(\b)$-module associated to $\mu$. Finally for $x,w \in W$, we use the Bruhat order, that is, $x \to w$ if and only if there is a reflection $t \in T$ so that $tx = w$ and $\ell(w) = \ell(x) + 1$. Here $T$ is the set of reflections of $W$ along the positive roots, which coincide with the set of all elements conjugates to some simple roots. We say that $x < w$ if there is a sequence $x \to \dots \to w$. We also recall the following facts:

\begin{proposition}[\cite{humphreys}]\label{prop1}
For all $\mu, \lambda \in \h^*$ we have $\dim \Hom_{\g}(M(\mu), M(\lambda)) \leq 1$. Moreover, if $\lambda$ is dominant, then for any $x,w \in W$ a morphism $M(w \cdot \mu) \to M(x \cdot \mu)$ exists if and only if $x < w$. Such a morphism is always an embedding.
\end{proposition}

\begin{proposition}[\cite{BGG}]\label{prop:bruhatsquare}
If $w,w' \in W$ are such that there is $x \in W$ with $w' \to x \to w$, then there are exactly two such elements, say $x$ and $y$. We call such a quadruple of elements a \emph{square} in $W$ and denote it by $(w',x,y,w)$.
\end{proposition}

\begin{proposition}[\cite{BGG}]\label{prop:splitsquare}
Let $\alpha$ be a simple root, $\beta$ be a positive root and $x,y \in W$. The first diagram exists if and only if the second diagram does: 

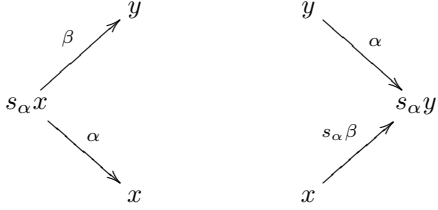
\[ \xymatrix{
    & y   &   & y \ar[rd]^{\alpha}  & \\
      s_{\alpha}x \ar[ru]^{\beta} \ar[rd]^{\alpha}    &  & &  &  s_{\alpha}y \\
      & x & & x \ar[ru]^{s_{\alpha}\beta}
    }\] 

\end{proposition}

\noindent We also require the following lemma: 

\begin{lemma}\label{prop:simplerootsquare}
For each edge $x \overset{t}{\to} w$, there is a square $(w',x,y,w)$ and a simple reflection $s$ so that either $w = sx$ or $w = sy$.
\end{lemma}

\begin{proof}
The lemma is obvious if $w = w_0$ so we assume $w \neq w_0$. Since we assumed $w \neq w_0$ there is a simple root $\alpha$ so that $y:= s_{\alpha} w \to w$. Taking $\beta = s_{\alpha} \gamma$ (where $\gamma$ is the positive root corresponding to $t$) proves the existence of the right part of the diagram in \ref{prop:splitsquare}, and the proposition finishes the proof. 
\end{proof}

Finally we recall the BGG Theorem, due to Bernstein-Gelfand-Gelfand  and Rocha-Caridi : 

\begin{theorem}\label{BGGres}\cite{BGG}\cite{rocha}
Let $\lambda \in P^+$, then there is an exact sequence 
\[
    0 \to M(w_0 \cdot \lambda ) \to  \dots \to \bigoplus_{\ell(w) = k} M(w \cdot \lambda) \to \dots \to M(\lambda) \to L(\lambda) \to 0
\]
\end{theorem}

\noindent We briefly sketch the proof. See \cite{humphreys} for a more complete sketch, and \cite{BGG} and \cite{rocha} for a complete proof. 

\begin{proof}[Proof (sketch)] Using translation functors one can assume that $\lambda = 0$. Then the complex $D_k := \text{Ind}_{U(\b)}^{U(\g)} (\bigwedge^k(\g/\b))$ resolves $L(0)$. Next, the subcomplex $C_k := D_k^{\chi_0}$ (i.e taking the part where the center acts by zero) is still exact. The main point is that each $C_k$ has a filtration where each Verma module appears exactly once. It is shown in \cite{rocha} that such a filtration necessary splits, giving the BGG theorem. 
\end{proof}

\subsection{Maps in the BGG complex}

In light of proposition \ref{prop1} and \ref{BGGres}, the maps in the BGG complex exactly correspond to pairs $(x,w)$ with $x,w\in W$, such that $x \to w$, i.e there is a reflection $t \in W$ so that $w = tx$ and $\ell(w) = \ell(x)+1$. Hence, the maps in the BGG complex exactly corresponds to the edges of the Bruhat graph $\mathscr B$ of the corresponding Weyl group $W$. 

We recall the definition of the Bruhat graph: the vertices of $\mathscr B$ are given by $W$, and there is an edge from $x$ to $y$ if and only if $x \to w$, i.e there is a reflection $t \in W$ with $w = tx$ and $\ell(w) = \ell(x) + 1$. For example, for $\sl_3$ the BGG complex associated to $\lambda$ is represented by the following diagram:

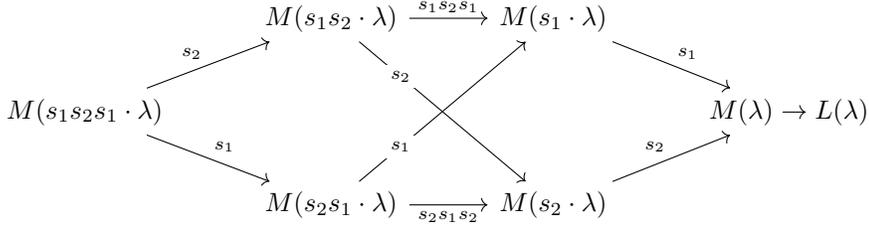
\[
    \begin{tikzcd}[column sep=3em]
        & M(s_1s_2 \cdot \lambda) \ar[r,"s_1s_2s_1"] \ar[rdd,"s_2" description, near start]& M(s_1 \cdot \lambda ) \ar[rd,"s_1"]  &  &   \\
        M(s_1s_2s_1 \cdot \lambda )  \ar[rd,"s_1"] \ar[ru,"s_2"] &  &  & M(\lambda) \to L(\lambda)  \\
        & M(s_2s_1 \cdot \lambda ) \ar[r, "s_2s_1s_2"'] \ar[ruu,"s_1" description, near start] & M(s_2 \cdot \lambda) \ar[ru,"s_2"] & &
    \end{tikzcd}
\]

Here each column represent a term of the BGG complex, and each edge $\sigma(x,w)$ corresponds to an element $t \in W$ so that $tx = w$. For example, for the lower horizontal arrow we have $x = s_2, w = s_2s_1$ and $t = s_2s_1s_2$.

Now let us describe in more detail the maps associated to edges in the Bruhat graph. A non-zero map $M(w \cdot \lambda ) \to M(x \cdot \lambda)$ is injective (again see \cite{humphreys}), so we can write $M(w \cdot \lambda) = U(\n)u$ where $u$ is a highest weight vector of weight $w \cdot \lambda$. The map $M(x \cdot \lambda) \to M(w \cdot \lambda)$ is determined by the image of $u$, which is a highest weight vector of weight $w \cdot \lambda - x \cdot \lambda$, so it can be written as $\mathcal{F} u$ for a unique $\mathcal{F} := \mathcal{F}(x,w) \in U(\n)[{x \cdot \lambda - w \cdot \lambda}]$. We emphasize that $\mathcal{F}(x,w)$ depends on $\lambda$. Explicitly computing all these elements $\mathcal{F}(x,w)$ is an important part of the algorithm. An easy case is when $w \cdot \lambda - x \cdot \lambda$ is a multiple of a simple root $\alpha_i$ (i.e when the element $t \in W$ associated to the edge $\sigma(x,w)$ correspond to a simple reflection). In this case, the element $\mathcal F$ will simply be a scalar multiple of $f_i^{1 + \langle w \cdot \lambda , \alpha_i^{\vee} \rangle }$.  

Now assume that we found all the elements $\mathcal{F}(x,w)$. These elements are well-defined up to scalar multiple. To obtain a complex, we should pick scalars for each edges to ensure the equation $d^2 = 0$. In view of propositions \ref{prop1} and \ref{prop:bruhatsquare}, to check that $d^2 = 0$ it is enough to check it `square-wise'. That is, for each square $w \to x \to w', w \to y \to w'$ we require the equality $\mathcal{F}(w,x)\mathcal{F}(x,w') + \mathcal{F}(w,y)\mathcal{F}(y,w') = 0$. In practice it is easier to solve the equation $\mathcal{F}(w,x)\mathcal{F}(x,w') = \mathcal{F}(w,y)\mathcal{F}(y,w')$ and then assign signs $\sigma(w,w')$ to each edge so that $\sigma(w,x)\sigma(x,w') + \sigma(w,y)\sigma(y,w') = 0$. Note that then the maps $\sigma(x,w)\mathcal{F}(x,w)$ will form a complex. The two following results ensure that the BGG resolution is in fact exact, and moreover unique:

\begin{theorem}\label{signs}\cite{BGG}
For any choice of scalars such that $d^2 = 0$, the corresponding BGG complex is exact. 
\end{theorem}

\begin{theorem}\cite{mm}
Different choices of scalars give isomorphic complexes.
\end{theorem}

So we can just pick an arbitrary choice of scalars such that $d^2 = 0$. Finally we explain how we compute the maps. For each square, we can recursively solve the equation $fg = hk$ where we know all but one map. We surely know the first column since each map correspond to a monomial. For any edge $t$ we can then find a square where the opposite edge is a simple reflection by lemma \ref{prop:simplerootsquare} (hence we know all but one map and can compute the map corresponding to $t$). Once we solved each system, we just distribute signs for each edge so that each signed square has $1$ or $3$ negative signs. A more precise description will be given in section~\ref{sec:Algorithm}. 

As an example we compute the first non-trivial map in the BGG resolution for $\g = \g_2$ and $\lambda = 0$ (note that the maps depend on the weights in general). We write $\alpha_1$ for the short root and $\alpha_2$ for the long root, and obtain $s_1 \cdot 0 = -\alpha_1, s_2 \cdot 0 = - \alpha_2, s_2s_1 \cdot 0 = - \alpha_1 - 2 \alpha_2$ and $s_1 s_2 \cdot 0 = - 4 \alpha_1 - \alpha_2$. Therefore the beginning of the complex is:

\[ \xymatrix{
     \ar[r]& M(-4 \alpha_1 - \alpha_2) \ar[r] \ar[rdd]& M(-\alpha_1) \ar[rd]  &  &   \\
      \dots   &  &  & M(0) \to L(0)  \\
      \ar[r]& M(- \alpha_1 - 2\alpha_2) \ar[r] \ar[ruu] & M(-\alpha_2) \ar[ru] & &
    }
\]

Write $M(0) = U(\n)v_0$ where $v_0 = 1$ and $M(-\alpha_1) = U(\n)v_{-\alpha_1}$. Looking at the weight, it is clear that the map $M(-\alpha_1) \to M(0)$ is given by $v_{-\alpha_1} \mapsto f_1v_0$, up to scaling. Similarly, the map $M(-4\alpha_1 - \alpha_2) \to M(-\alpha_2)$ is given by $v_{-4\alpha_1 - \alpha_2} \mapsto f_1^4 v_{-\alpha_2}$ up to scaling. Since we require the upper parallelogram to commute, we need to solve the equation $f_1^4f_2 = \mathcal F f_1$ where $\mathcal F$ is the unknown. For this particular case, the solution follows easily from the Serre relation $\ad(f_1)^4(f_2)=0$, giving $\mathcal F = 4f_1^3f_2 - 10f_1^2f_2f_1 + 4 f_1f_2f_1^2 - f_2f_1^3$. Similarly, the map $M(-\alpha_1 - \alpha_2) \to M(-\alpha_1)$ is obtained from the other Serre relation $\ad(f_2)^2(f_1) = 0$. A possible choice of signs for $\g_2$ is shown below (solid arrows correspond to $+$ and dashed arrows to $-$) :
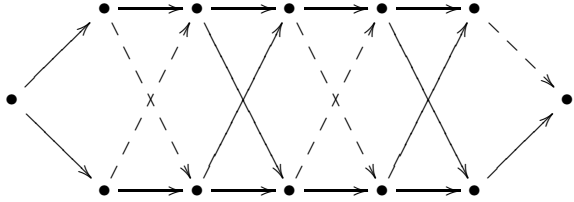
\[ 
\xymatrix{ 
    & \bullet  \ar[r] \ar@{-->}[rdd] & \bullet \ar[r] \ar[rdd] & \bullet \ar[r] \ar@{-->}[rdd]&  \ar[r] \bullet \ar[rdd] & \ar@{-->}[rd] \bullet &\\ 
    \bullet \ar[ru] \ar[rd] & & & & & & \bullet  \\ 
    & \bullet  \ar[r] \ar@{-->}[ruu] & \bullet  \ar[r] \ar[ruu] & \bullet  \ar[r] \ar@{-->}[ruu] & \bullet  \ar[r] \ar[ruu] & \bullet \ar[ru] & 
}
\]

\section{Sheaf cohomology of flag varieties and the BGG complex}\label{sec:Flag}

Following \cite{LQ} we explain how to use the BGG resolution to compute the cohomology of flag varieties.

\subsection{Flag varieties}
\label{sec:FlagVar}

Let $G$ be a simple complex algebraic group, $B$ a Borel subgroup and $E$ a $B$-module. We consider the vector bundle $\mathcal E := G \times_B E$ on the flag variety $X = G/B$. Since $\mathcal E$ is $G$-equivariant, the cohomology groups $H^{\bullet}(G/B, G \times_B E)$ are naturally $G$-modules. In \cite{LQ} these $G$-modules are computed using the BGG complex. In this section we explain their method. To compute $H^{\bullet}(X, \mathcal E)$, we only need to compute  $\Hom_G(L(\lambda), H^{\bullet}(X, \mathcal E))$ for all $\lambda\in P^+$. To obtain a list of dominant weights $\lambda$ that might contribute to $H^{\bullet}(X, \mathcal E)$, we use a filtration of $E$ given by the $\n$-action. The composition factors $\n^kE/\n^{k+1}E$ are direct sums of $1$-dimensional weight spaces. It follows that $\mathcal E$ has a filtration with compositions factors isomorphic to direct sum of line bundles. At this point it is useful to recall the Borel-Weil-Bott theorem : 

\begin{definition}
We define $\mathscr L_{\lambda}$ be the homogeneous line bundle corresponding to the $1$-dimensional $B$-module $\C_{-\lambda}$, (so $H$ acts by the character $\chi_{-\lambda} : H \to \C$ and $U = B/T$ acts by the identity).
\end{definition}

\begin{remark}
The sign is just for convention, and makes the Borel-Weil-Bott theorem simpler to state.
\end{remark}

\begin{theorem}[Borel-Weil-Bott]
Assume $\lambda \in P$ is dot-singular, then $H^i(X,\mathscr L_{\lambda}) = 0$ for all $i \in \mathbb N$. If $\lambda$ is dot-regular, let $w \in W$ the unique element so that $w \cdot \lambda$ is dominant. Then, $H^i(X, \mathscr L_{\lambda}) = L(\lambda)$ if $i = \ell(w)$ and $0$ else. 
\end{theorem}

If $E$ was our initial homogeneous vector bundle filtered by direct sum of homogeneous line bundles, it follows that there is a surjection from the direct sum of cohomology groups of these line bundles onto $H^{\bullet}(X, \mathcal E)$. Hence, by the Borel-Weil-Bott theorem, the only representations that can appear in $H^{\bullet}(X, \mathcal E)$ are given by the dot-orbits of the set $\wt(E)$.  Hence it is natural to try to find the kernel of this surjection to obtain the cohomology $H^{\bullet}(X, \mathcal E)$, but it boils down to compute the map in the associated spectral sequence which is not trivial. As an alternative road, it is possible to use the BGG resolution to explicitly compute the cohomology as was done in \cite{LQ}. The main role is played by Bott's theorem relating the sheaf cohomology on X with the Lie algebra cohomology.

\begin{theorem}\cite{bott}
Let $X = G/B$ and $\mathcal E = G \times_B E$ for a $B$-module $E$. For each $\lambda \in P^+$ there is a vector space isomorphism $\Hom_G(L(\lambda), H^{\bullet}(X, \mathcal E)) \cong H^{\bullet}(\b, \h, \Hom_B(L(\lambda),E) )$.
\end{theorem}

Here, $H^{\bullet}(\b,\h,\Hom_B(L(\lambda),E))$ is the $\h$-equivariant Chevalley-Eilenberg cohomology of $\b$ with coefficients in $E \otimes L(\lambda)^*$. By definition, this is the cohomology of the complex $\Hom(\wedge^{\bullet} (\b/\h)^* \otimes \Hom_B(L(\lambda),E))^{\h}$. We also need the following lemma, originally stated in \cite{bott} :

\begin{lemma}\cite{bott}
For a $\b$-module $F$ there is an isomorphism $H^{\bullet}(\b, \h, F) \cong H^{\bullet}(\n,F)^{\h}$.
\end{lemma}

F ollowing \cite{LQ}, we will explain how to compute the cohomology groups $H^{\bullet}(G/B, \mathcal E)$ using the BGG complex. We need to compute $H^{\bullet}(\b,\Hom_B(L(\lambda, E))^{\h} \cong \text{Ext}^{\bullet}(\Bbb C, L_{\lambda} \otimes E^*)^{\h}$. By definition, the latter is computed by picking an $\h$-graded resolution of $L(\lambda) \otimes E^*$ as a $U(\n)$-module. Tensoring the BGG resolution for $L(\lambda)$ with $E^*$ exactly gives a projective resolution for $L(\lambda) \otimes E^*$. Hence the complex that computes $H^{\bullet}(\n, \Hom_B(L(\lambda),E))^{\h}$ is given by $\Hom(M_{\bullet}(\lambda) \otimes E^*,\Bbb C)^{\h}$. To describe the terms in the previous complex, we use that for a Verma module $M(\mu)$ we have $\Hom_{U(\n)}(M(\mu), E)^{\h}  \cong E[\mu]$. Hence the $k$-th term of the complex computing $\Hom(L(\lambda), H^{\bullet}(X,\mathcal E))$ is given by $\bigoplus_{\ell(w) = k} E[ w \cdot \lambda]$.

\begin{definition} The BGG complex associated to $E$ and $\lambda$, written $\BGG^{\bullet}(E, \lambda)$, is the complex $\Hom(M_{\bullet}(\lambda) \otimes E^*,\Bbb C)^{\h}$. 
\end{definition}

By the discussion before, there is an identification $\BGG^k(E, \lambda) \cong \bigoplus_{\ell(w) = k} E[w \cdot \lambda]$. In this setting, the maps $E[x \cdot \lambda] \to E[w \cdot \lambda]$ are just given by multiplication by $\mathcal{F}(x,w)$. To summarize, the multiplicities $Hom_G(L(\lambda), H^{\bullet}(X, \mathcal E))$ can be computed using the BGG complex as explained. 

\subsection{Examples}

Below we present two examples that can be computed by hand, using the algorithm previously described. The computations were done by hand and confirmed by our program. Our first example is the cohomology of the flag variety $G/B$ for $G = G_2$. Unlike the traditional approach via the Chevalley-Eilenberg complex, the computation using the BGG resolution is straightforward. In our second example (which is more involved), we compute the Hochschild cohomology of the complete flag variety $X = G/B$ for $G = \SL_4$. The cohomology of the $G_2$-flag variety is well-known but our computation of $HH^{\bullet}(G/B)$ ($G = \SL_4$) seems to be new.

\subsubsection{Cohomology of $G/B$, $G = G_2$}

Let $X$ be the complete flag variety of type $G = G_2$. We would like to compute $H^k(X, \C)$. By the Hodge decomposition and the fact that Schubert classes are algebraic, this group is isomorphic to $H^q(X, \Omega_X^q)$. Moreover, thanks to Poincaré duality we only need to compute it for $q=1,2,3$. Finally, it is a well-known fact that the $\g$-structure on $H^q(X, \Omega_X^q)$ is trivial, hence we can focus solely on the multiplicity of the trivial representation. 

Let $\alpha_1$ denote the short root and $\alpha_2$ the long root. We pick a Chevalley basis of $\n$ with elements $f_1, f_2, f_{12}, f_{112}, f_{1112}$ and $f_{11122}$ where the subscript indicates the weight. The following diagram represents the dot-orbit of $0$, indexed by the Weyl group of $G_2$. It shows which weights will appear in the BGG complex associated to $\lambda =0$.

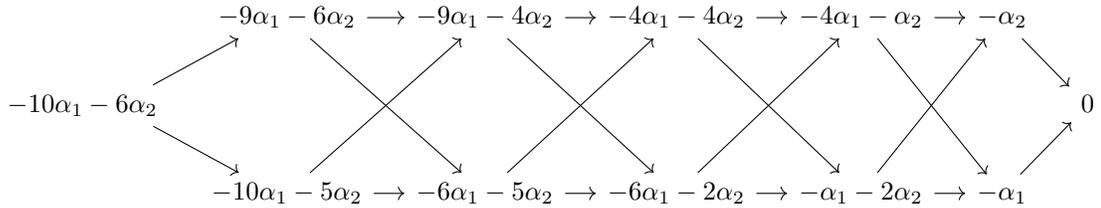
\[
\begin{tikzcd}[column sep=small]
    & -9\alpha_1 - 6\alpha_2  \ar[r] \ar[rdd]  & -9\alpha_1 - 4 \alpha_2 \ar[r] \ar[rdd] & -4\alpha_1 - 4\alpha_2 \ar[r] \ar[rdd]  &   -4\alpha_1 - \alpha_2 \ar[r] \ar[rdd]  &  -\alpha_2 \ar[rd] &\\ 
    -10 \alpha_1 - 6\alpha_2 \ar[ru, start anchor={[xshift = 2ex]}] \ar[rd,  start anchor={[xshift = 2ex]}]  & & & & & & 0  \\ 
    & -10\alpha_1 - 5 \alpha_2  \ar[r] \ar[ruu] & -6\alpha_1 - 5 \alpha_2  \ar[r] \ar[ruu] & -6\alpha_1 - 2 \alpha_2  \ar[r] \ar[ruu] & -\alpha_1 - 2 \alpha_2  \ar[r] \ar[ruu] & -\alpha_1 \ar[ru] & 
\end{tikzcd}
\]

In order to compute $H^1(X, \Omega_X^1)$ we use that $\Omega_X^1 = G \times_B \n$. Hence, by our previous discussion we need to compute $H^1\BGG^{\bullet}(\n,0)$. We have $\BGG^2(\n, 0) = \n[-\alpha_1 - 2\alpha_2] \oplus \n[-4\alpha_1-\alpha_2]$. However it's clear that these weight spaces are zero, and it's obvious that $\BGG^0(\n,0) = 0$. Hence we get $H^1(X, \Omega_X) = L(0)^{\oplus 2}$, generated by $f_1$ and $f_2$. That was the expected answer since the Picard group of $G_2/B$ is generated by the divisors associated to the line bundles $\mathscr L_{\alpha_1}$ and $\mathscr L_{\alpha_2}$.  Let us compute $H^2(X, \C) = H^2(X, \Omega^2_X)$. This time our vector bundle is $G \times_B \wedge^2\n$, so we would like to compute $H^2\BGG^{\bullet}(\wedge^2\n,0)$. Clearly, we have $\wedge^2\n[-\alpha_1] = \wedge^2\n[-\alpha_2] = 0 $ and $\wedge^2\n[-6\alpha_1-2\alpha_2] = \wedge^2\n[-4 \alpha_1 -4\alpha_2] = 0$ by direct inspection. So our cohomology is generated by $\wedge^2 \n[-4\alpha_1 - \alpha_2]$ (generated by $f_1 \wedge f_{1112}$) and $\wedge^2\n[-\alpha_1 -2\alpha_2]$ (generated by $f_1 \wedge f_{12}$). As a last example, let us compute $H^3(X, \Omega_X^3)$. Once more it is clear that $\wedge^3\n[-9\alpha_1 - 4 \alpha_2] = 0 = \wedge^3\n[-6\alpha_1 - 5 \alpha_2]$, and similarly $\wedge^3 \n[- \alpha_1 - 2 \alpha_2] = 0 = \wedge^3\n[-4 \alpha_1 - \alpha_2]$. It follows that $H^3(X, \Omega_X^3) \cong \wedge^3\n[-6 \alpha_1 - 2 \alpha_2] \oplus \wedge^3\n[-4 \alpha_1 - 4 \alpha_2]$, generated by $ f_1 \wedge f_{12} \wedge f_{112}$ and $f_2 \wedge f_{12} \wedge f_{11122}$ respectively. Other computations are similar. We present the final result in a table~\ref{tab:G2cohom}.

\begin{table}[!htb]
\centering
\begin{tabular}{ r| c l } 
 $q$ & $\dim H^q(X, \Omega_X^q)$ & Explicit BGG generators \\ 
 \hline 
 $0$ & $1$ & $1$ \\ 
 $1$ & $2$ & $f_1$, $f_2$ \\ 
 $2$ & $2$ & $f_1 \wedge f_{1112}$, $f_2 \wedge f_{12}$ \\ 
 $3$ & $2$ & $f_1 \wedge f_{112} \wedge f_{1112}$, $f_2 \wedge f_{12} \wedge f_{11122}$ \\ 
 $4$ & $2$ & $f_2 \wedge f_{12} \wedge f_{112} \wedge f_{11122}$ \\ 
 & &   $f_1 \wedge f_{112} \wedge f_{1112} \wedge f_{11122}$ \\ 
 $5$ & $2$ & $f_2 \wedge f_{12} \wedge f_{112} \wedge f_{1112} \wedge f_{11122}$ \\ 
 & & $f_1 \wedge f_{12} \wedge f_{112} \wedge f_{1112} \wedge f_{11122}$ \\ 
 $6$ & $1$ & $f_1 \wedge f_2 \wedge f_{12} \wedge f_{112} \wedge f_{1112} \wedge f_{11122}$ 
\end{tabular}\caption{Cohomology of $G/B$ for $G=G_2$}\label{tab:G2cohom}
\end{table}

One can notice the Poincaré duality through the BGG complex. For example $f_1$ and $f_2 \wedge f_{12} \wedge f_{112} \wedge f_{1112} \wedge f_{11122}$ are Poincaré dual to each other. This duality is specific to the case where $\lambda = 0$. In the next example, such duality does not appear.

\subsubsection{Hochschild cohomology of $G/B, G = \SL_4$}

\noindent Recall that the Hochschild cohomology of a smooth scheme $X$ over a field is defined as $HH^{\bullet}(X) := Ext^{\bullet}_{X \times X}(\O_{\Delta}, \O_{\Delta})$, where $\Delta \subset X \times X$ is the diagonal. It is a famous theorem by Hochschild-Kostant-Rosenberg that there is an isomorphism $HH^{\bullet}(X) \cong \bigoplus_{i + j = \bullet} H^i(X, \wedge^j TX)$. If $X = G/B$ (for type $A_4$) it turns out that if $i>0$ then $H^i(X, \wedge^jTX) = 0$. This property holds as well for the grassmannians and the projective spaces. This property was first noticed by Pieter Belmans in his PhD thesis \cite{belmans}. He asked if the property holds for other flag varieties. We confirm this vanishing property for any $G/P$ where $G$ is of type $A_4$. In this paragraph we compute the Hochschild cohomology of $G/B$ for $G$ of type $A_3$. It is a classical result (see \cite{penrose}) that $H^0(G/P, T_{G/P}) = \g$ if $G$ is of type $A$.  Hence we just need to compute $H^0(X, \wedge^j T_X)$ for $j = 2,3,4,5,6$. We have $T_X = \SL_4 \times_B \u$ where $\u \cong \g/\b$ is the Lie subalgebra generated by $e_1,e_2,e_3$. In order to not duplicate computations we will use the symmetry exchanging $\alpha_1$ and $\alpha_2$. Below we list  the cohomology $H^0(X, \wedge^jTX)$ for each $j$.

\begin{itemize}
    \item $H^0(X,\mathcal O_X) = \mathbb C$
    \item $H^0(X,TX)=L(\alpha_1+\alpha_2+\alpha_3)$
    \item $H^0(X,\wedge^2TX)=L(\alpha_1+\alpha_2+\alpha_3)\oplus L(2\alpha_1+2\alpha_2+\alpha_3)\oplus L(\alpha_1+2\alpha_2+2\alpha_3)$
    \item $H^0(X,\wedge^3TX)=L(\alpha_1+\alpha_2+\alpha_3)\oplus L(2\alpha_1+2\alpha_2+\alpha_3)\oplus L(\alpha_1+2\alpha_2+2\alpha_3)\oplus L(2\alpha_1+2\alpha_2+2\alpha_3)\oplus L(\alpha_1+2\alpha_2+\alpha_3)^{\oplus 2}\oplus L(3\alpha_1+2\alpha_2+\alpha_3)\oplus L(\alpha_1+2\alpha_2+3\alpha_3)\oplus L(2\alpha_1+3\alpha_2+2\alpha_3)$
    \item $H^0(X,\wedge^4TX)=L(2\alpha_1+2\alpha_2+\alpha_3)\oplus L(\alpha_1+2\alpha_2+2\alpha_3)\oplus L(2\alpha_1+2\alpha_2+2\alpha_3)^{\oplus 2}\oplus L(2\alpha_1+3\alpha_2+2\alpha_3)^{\oplus 2}\oplus L(3\alpha_1+3\alpha_2+2\alpha_3)\oplus L(2\alpha_1+3\alpha_2+3\alpha_3)\oplus L(2\alpha_1+4\alpha_2+2\alpha_3)$
    \item $H^0(X,\wedge^5TX)=L(2\alpha_1+3\alpha_2+2\alpha_3)\oplus L(3\alpha_1+3\alpha_2+2\alpha_3)\oplus L(2\alpha_1+3\alpha_2+3\alpha_3)\oplus L(3\alpha_1+3\alpha_2+3\alpha_3)\oplus L(3\alpha_1+4\alpha_2+2\alpha_3)\oplus L(2\alpha_1+4\alpha_2+3\alpha_3)$
    \item $H^0(X,\wedge^6TX)=L(3\alpha_1+4\alpha_2+3\alpha_3)$
\end{itemize}

We consider the computation for $H^0(X, \wedge^2TX)$ in some more detail. We fix the notation and take the following basis of $\h$: $h_1 = \text{diag}(1,-1,0,0)$, $h_2 = \text{diag}(1,0,-1,0)$ and $h_3 = \text{diag}(1,0,0,-1)$. For $i=1,2,3$ we take $f_i, e_i \in \sl_4$ so that $f_i$ is lower-triangular and each $(e_i,f_i,h_i)$ is a $\sl_2$-triple. We define $ f_{12} := [f_1,f_2], f_{23} := [f_2,f_3]$ and $f_{123} := [f_1,f_{23}] = [f_{12}, f_3]$ and $ e_{12} = [e_2,e_1], e_{23} = [e_3,e_2]$ and $e_{123} = [e_3, e_{12}]$. Let us compute the coadjoint action of $\b$ on $\u$. Note that the $\h$-action is the same as the adjoint action. Here we list the non-zero coadjoint actions of $\n$ on $\u$:
\[ \begin{tabular}{r | c c c}
 
  ${}$ & $e_{12}$ & $e_{23} $ & $e_{123}$  \\ 
  \hline $f_1$ & $e_2$ & $0 $ & $e_{23}$  \\ 
  $f_2$ & $-e_1$ & $e_{3} $ & $0$  \\ 
  $f_3$ & $0$ & $-e_2 $ & $-e_{12}$  \\ 
  $f_{12}$ & $0$ & $0 $ & $e_{3}$  \\ 
 
\end{tabular} \\
\] 

 Since $\u$ is nilpotent there is a natural filtration on it, and each composition factor splits as a direct sum of $1$-dimensional weight spaces. Geometrically it means that the vector bundle $TX$ has a filtration with direct sum of line bundles as composition factors. For example, in our case the composition factors for $TX$ are $\mathscr L_{\alpha_1} \oplus \mathscr L_{\alpha_2} \oplus \mathscr L_{\alpha_3}, \mathscr L_{\alpha_1 + \alpha_2} \oplus \mathscr L_{\alpha_2 + \alpha_3}$ and $\mathscr L_{\alpha_1 + \alpha_2 + \alpha_3}$, where for $\mu \in P$, $\mathscr L_{\mu}$ is the line bundle $G \times_B \C_{\mu}$. By considering the associated long exact sequences it is clear that there is a surjection $\bigoplus_{\mu \in \text{wt}(\wedge^2\u)} H^{0}(X, \mathscr L_{\mu}) \to H^0(X, \wedge^2TX)$. However, since line bundles might have higher cohomology, there could be some cancellation in the long exact sequence. Using the Borel-Weil theorem it is at least clear that dot-singular weights will not contribute, hence we can restrict ourselves to line bundles so that the corresponding weight is dot-regular. For each such weight $\lambda$, the cancellation will come from a dot-regular, non-dominant weight $\mu$ in the same dot-orbit as $\lambda$. Hence, if there is no such $\mu$ we know that the $L(\lambda)$-isotypic component of $H^0(X, \wedge^2TX)$ is given by the $\lambda$-weight space of $\wedge^2 \n$. In most cases we don't know how to compute the maps in the long exact sequence explicitly. However, we can still compute the multiplicity of $L(\lambda)$ in $H^0(X, \wedge^2TX)$ using the BGG complex.

The dominant, dot-regular weights appearing in $\wedge^2 \u$ are $\alpha_1 + \alpha_2 + \alpha_3$ with multiplicity $2$,  $\alpha_1 + 2\alpha_2 + \alpha_3$ with multiplicity $2$, $2\alpha_1 + 2\alpha_2 + \alpha_3$ and $\alpha_1 + 2\alpha_2 + 2\alpha_3$ with multiplicity 1. The dot-regular weights that are non-dominant are $\alpha_1 + \alpha_3$, $\alpha_1 + 2\alpha_2$ and $2\alpha_2 + \alpha_3$. We have $s_2 \cdot (\alpha_1 + \alpha_3) = \alpha_1 + \alpha_2 + \alpha_3$ and $s_3 \cdot (\alpha_1 + 2 \alpha_2) = \alpha_1 + 2 \alpha_2 + \alpha_3$. This means that these two dominant weights might appear in the cohomology.

To compute the multiplicity of $L(\alpha_1 + \alpha_2 + \alpha_3)$ in $H^1(X, \wedge^2TX)$, we use the BGG complex $\BGG^{\bullet}(\wedge^2\u, \alpha_1 + \alpha_2 + \alpha_3)$. Since $s_1 \cdot (\alpha_1 + \alpha_2 + \alpha_3) = -\alpha_1 + \alpha_2 + \alpha_3$ and $s_2 \cdot (\alpha_1 + \alpha_2 + \alpha_3) = \alpha_1 + \alpha_3$. We see that $\BGG^1(\wedge^2\u, \alpha_1 + \alpha_2 + \alpha_3) = \C\{e_1 \wedge e_3\}$. Cleary $\BGG^0(\wedge^2\u, \alpha_1 + \alpha_2 + \alpha_3)$ is spanned by $e_1 \wedge e_{23}$ and $e_{12} \wedge e_3$. The differential given by the coadjoint action is $x \mapsto f_2 x$ and it is clearly surjective, hence $\Hom(L(\alpha_1 + \alpha_2 + \alpha_3),H^1(X, \wedge^2 TX)) = 0$ and $\Hom(L(\alpha_1 + \alpha_2 + \alpha_3),H^0(X, \wedge^2 TX)) = \C$. 

To compute $\Hom(L(\alpha_1 + 2\alpha_2 + \alpha_3), H^1(X, \wedge^2 TX))$ we also use the BGG complex: 
\begin{equation}
    \wedge^2 \u[\alpha_1 + 2\alpha_2 + \alpha_3] \to   \wedge^2 \u[2 \alpha_2 + \alpha_3] \oplus   \wedge^2 \u[\alpha_1 + 2 \alpha_2] \to 0
\end{equation}

A basis of $\wedge^2 \u[\alpha_1 + 2\alpha_2 + \alpha_3]$ is given by $e_{12} \wedge e_{23}$ and $ e_2 \wedge e_{123}$ and $ \wedge^2 \u[2 \alpha_2 + \alpha_3] \oplus   \wedge^2 \u[\alpha_1 + 2 \alpha_2]$ has basis spanned by $e_2 \wedge e_{12}$ and $e_2 \wedge e_{23}$. The differential is 
\begin{align*}
   d(e_2 \wedge e_{123})  &= e_2 \wedge e_{23} - e_2 \wedge e_{12}\\
    d(e_{2} \wedge e_{123}) &= e_2 \wedge e_{23} + e_2 \wedge e_{12}
\end{align*}
 meaning that $d$ is surjective so again there is no cohomology. We get  
 \[
     \Hom(L_{\alpha_1 + 2\alpha_2 + \alpha_3},H^1(X, \wedge^2 TX)) = 0 \text{, and } \Hom(L_{\alpha_1 + 2\alpha_2 + \alpha_3},H^0(X, \wedge^2 TX)) = \C.
 \]
 To summarize, $H^{\bullet}(X, \wedge^2TX)$ has no higher cohomology, and 
 \begin{equation}
    H^0(X, \wedge^2 TX) \cong L(\alpha_1 + \alpha_2 + \alpha_3) \oplus L(2(\alpha_1 + \alpha_2) + \alpha_3) \oplus L(\alpha_1 + 2(\alpha_2 + \alpha_3)).
 \end{equation}

\section{Description of the algorithm} \label{sec:Algorithm}

The algorithm can be divided in several steps. The first step consists of constructing the BGG complex, and computing the maps of the complex. Then we need to compute a distribution of signs on the edges to make the differential square to zero. Next we need a way to represent a basis of a $\mathfrak b$-module. Then we need a way to efficiently compute the $U(\mathfrak n)$-action on this basis, and finally we use the $U(\mathfrak n)$ action to construct a differential and compute its associated cohomology. 

\subsection{Maps in the BGG complex}
The first step in the algorithm is to construct the BGG complex and compute its maps (up to scalar). We begin by constructing the Bruhat graph $\mathscr B$ of the Weyl group $W$ associated to the simple Lie algebra $\mathfrak g$. This is done using the \cite{sagemath} implementation of the Weyl group.

For the remainder of the section we now fix an integral dominant weight $\lambda\in P^+$. We identify the weights $\wt(U(\mathfrak g))$ with $\mathbb{Z}^n$ in the basis of simple roots, where $n$ is the rank of $\mathfrak{g}$. Furthermore we identify
\[\wt(U(\mathfrak n))\simeq \mathbb{N}^n=\{(a_1,\ldots,a_n)\in \mathbb{Z}^n\,\mid\, a_i\leq 0\}.\]
We then fix the following ordering on the set of weights $\wt(U(\mathfrak n))=\mathbb{N}^n$. Write a weight $\xi=(\xi_1,\ldots \xi_n)$, then we say $\xi<\xi'$ if $\sum_i\xi_i<\sum_i\xi'_i$, or if this is equal then if the first non-zero coefficient of $\xi-\xi'$ is negative (this is the convention used in \cite{sagemath}). We then work with a PBW basis on $U(\mathfrak{n})$, determined by this ordering of the weights. Efficient multiplication in this basis is implemented in \cite{sagemath}.

For each edge $w\to w'$ in $\mathscr B$, there is a unique morphism of Verma modules $M(w' \cdot \lambda) \to M(w \cdot \lambda)$. As we discussed before, under the canonical identification $M(w \cdot \lambda) \cong U(\n)$, this morphism is determined by the image of the highest weight vector which is an element $\mathcal F(w,w')\in U(\mathfrak{n})$, well-defined up to scalar multiplication. Let $\xi=w'\cdot \lambda-w\cdot\lambda$, where $\xi\in \wt(U(\mathfrak{n}))$. Our goal is to determine these elements $\mathcal F(w,w')$ up to scalar multiple.

We let $U(\mathfrak{n})[\xi]$ denote the subspace of $U(\mathfrak n)$ consisting of elements with weight $\xi$. We note that if $\xi=m\alpha_i$ is a multiple of a simple root, then $U(\mathfrak n)[\xi]$ is one-dimensional and spanned by $(f_{\alpha_i})^m$. Hence $\mathcal F(w,w')=\gamma(f_{\alpha_i})^m$, for some scalar $\gamma$, which we can choose to be $\gamma=1$. If $\xi$ is not a multiple of a simple root, we find a basis for $U(\mathfrak{n})[\xi]$ by a simple combinatorial algorithm. 

Suppose we have a square $C\in \mathcal C$ in the Bruhat graph where we know three out of the four maps. We can then use commutativity to compute the fourth. Suppose without loss of generality that we have a square $(w',x,y,w)$ for which we know all maps except $\mathcal{F}(x,w')$; the other cases are completely analogous. 

\begin{equation}
    \begin{tikzcd}[column sep=5em]
        &x \ar[dr,dotted,"{\mathcal{F}(x,w')}"] \\ 
        w\ar[dr,"{\mathcal{F}(w,y)}"'] \ar[ur,"{\mathcal{F}(w,x)}"] &&w'\\
        &y  \ar[ur,"{\mathcal{F}(y,w')}"']
    \end{tikzcd}
\end{equation}

Let $\xi=w'\cdot \lambda-x\cdot\lambda$, and let $\xi'=x\cdot \lambda-w\cdot\lambda$. We have to find the unique $\mathcal{F}(x,w')\in U(\mathfrak{n})[\xi]$ such that 
\begin{equation}
    \mathcal{F}(x,w')\cdot \mathcal{F}(w,x) = \mathcal{F}(y,w')\cdot \mathcal{F}(w,y)
\end{equation} 
Let $\{f_i\}$ be a PBW basis of $U(\mathfrak{n})[\xi]$, then for each $i$ we compute $A_i=f_i\cdot \mathcal{F}(w,x)=\sum_jA_{ij}g_j$ where $\{g_j\}$ is a basis of $U(\mathfrak{n})[\xi]$. Note that in practice we do not have to compute a basis of $U(\mathfrak{n})[\xi']$ since we only use the monomials $g_j$ for which some $A_{ij}$ is non-zero. We then also define $b_j$ by $\mathcal{F}(y,w')\cdot \mathcal{F}(w,y)=\sum_jb_jg_j$. We then obtain $\mathcal{F}(x,w')$ by solving the integer linear problem $A^\top v=b$, i.e. $\mathcal{F}(x,w')=\sum_iv_if_i$. This problem can be solved with a package for exact integer linear algebra such as  \cite{linbox} or \cite{flint}.

To compute all the maps in the entire complex, we do the following. We start by identifying all the edges $(w,w')$ where $w'\cdot \lambda-w\cdot \lambda = m\alpha_i$, and set $\mathcal{F}(w,w') = (f_{\alpha_i})^m$. Then we iteratively find a square where we know three out of four maps, and compute the fourth. By lemma \ref{prop:simplerootsquare}, this allows us to compute all the maps.

\subsection{Signs in the complex}
Let $\mathscr B$ be the Bruhat graph of the Weyl group $W$ associated to the simple Lie algebra $\mathfrak g$. We wish to compute a distribution of signs $\sigma\colon E(\mathscr B)\to \{+1,-1\}$ on the edges $E(\mathscr B)$ of the Bruhat graph. Let $\mathcal{C}$ be the set of squares in the Bruhat graph (cf. prop. \ref{prop:bruhatsquare}), and call a square $C=(e_1,e_2,e_3,e_4)$ good (resp. bad) if the product of signs $\sigma(C):=\prod_i\sigma(e_i)$ is $-1$ (resp. $+1$). The aim is to find a choice of signs $\sigma$ such that all the squares $C\in\mathcal C$ are good. 

 We note that flipping the sign $\sigma(e)$ will reduce the number of bad squares if and only if: \begin{equation}\sum_{\{C|e\in C\}}\sigma(C)>0.\end{equation} This observation is the basis of the randomized greedy algorithm. The algorithm greedily flips signs of edges if it reduces the total number of bad squares. In practice this does not converge, therefore every time we run out of signs to flip in this way, we flip a number of signs completely at random. This algorithm is designed based on a heuristic, and we do not have a theoretical reason why this algorithm works better than sampling sign configurations at random. In practice this algorithm is never a bottleneck. 

\subsection{Constructing weight modules}

Let $M$ be a $\mathfrak b$-module with weight decomposition $M=\bigoplus_{\mu\in \mathbb{Z}^r} M[\mu]$, with $r$ the rank of $\mathfrak g$. To significantly reduce the sparsity of computations we assume that $M$ can be decomposed in the following way. Let $V_{ij}$ be a $\mathfrak b$ module, and suppose that $V_{ij}$ admits a basis compatible with the weight decomposition, and suppose that the $\mathfrak b$-action has integer coefficients in this basis. The integer coefficients are important to ensure exact computations. Then we require $M$ to be of the following form:
\begin{equation}
    \bigoplus_i\left(\bigotimes_j \sym^{n_{ij}}V_{ij} \otimes \wedge^{m_{ij}} V_{ij}\right)
\end{equation}
Moreover we assume this decomposition is compatible with the weight decomposition of $M$ in the sense that if each $X_i$ lies in a component of weight $\wt(X_i)$ in some $V_{ij}$, then
\begin{equation}
    X_1\otimes\dots\otimes X_n\in M\left[\sum_i\wt(X_i)\right].
\end{equation}

Note that if we set $V=M$, then $M$ trivially admits such a decomposition after choosing a basis compatible with the weight decomposition. However, using a decomposition like this we can compute the $U(\mathfrak n)$ action on $M$ much more efficiently, since we can compute the action on each $V_{ij}$ separately, greatly reducing the dimensionality. For example, if $V$ is $k$ dimensional, then the structure coefficients of the action on $\wedge^\ell V$ are the same as the action of just $V$ instead of storing structure coefficients for each basis element of $\wedge^\ell V$. As long as we have a basis of each $V_{ij}$ it is moreover very easy to compute a basis of $M$

In all our examples $V=\mathfrak g$, and $V_J$ corresponds to one of $\mathfrak g,\mathfrak n,\mathfrak u,\mathfrak b\subset \mathfrak g$ (or their parabolic counterparts) with either the adjoint or coadjoint action of $\mathfrak n$, and the basis is the Chevalley basis.

\subsection{Computing the $U(\mathfrak n)$ action}
Next we describe how to compute the $U(\mathfrak n)$ action on a basis of $M[\mu]$. This is the most technical part of the algorithm, and we feel that it is best explained through an example. Let us take $\mathfrak g = \sl_3$, and consider the module $M=\sym^2\mathfrak n$. We will compute the action of $f_1$ and $f_2f_1$ on the entire module. Note that $\mathfrak n$ has basis $f_1,f_2,f_{12}$ which we index by $0,1,2$. Then the only non-zero structure coefficients of the $\mathfrak n$ action are $C_{01}^2=1,C_{10}^2=-1$, corresponding to $[f_1,f_2]=f_{12}$, $[f_2,f_1]=-f_{12}$. We then start with matrix $B$ given by:
\[
    B = \left(\begin{array}{@{}cc|c|c@{}}
        0 & 0 & 0 & 1\\
        0 & 1 & 1 & 1\\
        0 & 2 & 2 & 1\\
        1 & 1 & 3 & 1\\
        1 & 2 & 4 & 1\\
        2 & 2 & 5 & 1
    \end{array}\right)
\]
Here the first columns rows are the indices enumerating a basis of $\sym^2\mathfrak n$, the third column gives the index of each basis element so that we can keep track of where each row came from initially. The final column gives the coefficients assigned to each basis element. Our procedure then gives the following $f_1\cdot B$
\[
 f_1\cdot B = \left(\begin{array}{@{}cc|c|c@{}}
        2 & 1 & 3 & 1\\
        2 & 2 & 4 & 1\\
        0 & 2 & 1 & 1\\
        1 & 2 & 3 & 1\\
    \end{array}\right),
\]
where the first two rows correspond to the action on the first column, and the last two rows to the action on the second column. If we then act by $f_2$ we get
\[
 f_2f_1\cdot B = \left(\begin{array}{@{}cc|c|c@{}}
        2 & 2 & 1 & -1
    \end{array}\right)
\]
This corresponds to the fact that the only non-trivial $f_2f_1$ action is given by 
\[
 f_2f_1(f_1\odot f_2) = -f_{12}\odot f_{12}
\]
If we instead were to just compute the $f_1$ action, then we would have to first of all sort each row of $f_1\cdot B$, then reorder the rows and merge duplicate entries to obtain
\[
 f_1\cdot B = \left(\begin{array}{@{}cc|c|c@{}}
        0 & 2 & 1 & 1\\
        1 & 2 & 3 & 2\\
        2 & 2 & 4 & 1
    \end{array}\right)
\]
This corresponds to the fact that the only non-trivial $f_1$ actions are given by 
\[
 f_1\cdot (f_1\odot f_2) = f_1\odot f_{12},\qquad f_1\cdot (f_2\odot f_2) = 2f_2\odot f_{12},\qquad f_1\cdot (f_2\odot f_{12}) = f_{12}\odot f_{12}
\]

If we have multiple monomials, we compute the action of each and concatenate the results in a big matrix. We then merge the duplicate entries and add the coefficients to obtain the action. When dealing with alternating products of modules one also needs to keep track of signs. Otherwise this procedure easily translates to compute the $U(\mathfrak{n})$ action on a basis of the general type of module described in the previous section. 
    
\subsection{Computing the cohomology}

By the discussion at the beginning of section~\ref{sec:FlagVar}, we simply need to explain how the computer algorithm computes the multiplicity of $L(\lambda)$ in the cohomology. However, let us emphasize that the algorithm computes the set of $\lambda$ so that $\dim_G(L(\lambda), H^{\bullet}(G/B, \mathcal E)) \neq 0$ and computes it for each $\lambda$. Fix an integral dominant weight $\lambda\in P^+$ and $\mathfrak b$-module $M$. We will describe how to compute the cohomology $H^i(\BGG_{\bullet}(M,\lambda))$ of the BGG resolution of $M$. The spaces $\BGG_{i}(M,\lambda)$ are given by $\bigoplus_{\ell(w)=i}M[w\cdot \mu]$, and the differential is given by 
\begin{equation}
    d_i = \hspace{-1em}\sum_{\substack{\ell(w)=i\\w\to w'\in \mathscr B}}\hspace{-1em} \sigma(w,w)\mathcal F(w,w').
\end{equation}
To compute this differential we compute the action of the $\mathcal F(w,w')$ as described in the previous section. This gives a sparse integer matrix $\mathcal D_i$, with same non-zero entries as $d_i$. Because some rows of $d_i$ may be entirely zero, the sparse matrix $\mathcal D_i$ has the same rank as $d_i$, but can have smaller kernel dimension. We hence compute the rank of $\mathcal D_i$ and obtain the cohomology dimensions through the rank-nullity theorem. The rank of $\mathcal D_i$ can be computed through (dense or sparse) row reduction, as implemented in exact integer linear algebra packages such as \cite{flint} or \cite{linbox}. Since the rank of a matrix is numerically unstable, it is important that $\mathcal D_i$ has integer coefficients, so that exact rank computations are possible.

\section{Potential extensions of the algorithm}\label{sec:Extensions}

The implementation of the algorithm is currently a work in progress, and we expect the implementation to have additional features in the future. The code of the most computationally intensive parts will be improved, and parallelized where possible. Most of the parts of the algorithm fall under the category of embarrassingly parallel problems, and could therefore benefit significantly from parallelization. On the other hand it appears some crucial parts of the algorithm are bound by memory speed, and it is therefore unknown how much overall performance gain there will be from parallelization with a single machine. 

Currently the most computationally intensive part of the algorithm is computing the maps in the BGG complex. A better implementation could partially mitigate this, but it might also be possible to improve the algorithm itself in this regard. The algorithm solves division problems in $U(\mathfrak n)$ in a relatively naive way, and a smarter division algorithm could significantly reduce the run time. Specifically for type $A_n$ explicit formulas for the BGG maps have been derived \cite{MFF,Xiao}. Implementing this would improve speed for type $A_n$. For a large class of weights, \cite{Xiao2} also derived formulas for type $C_n$, but the problem remains unsolved for general type. Another approach is given by Lutsyuk \cite{Lutsyuk}. They derived a recursion formula which is not immediately useful for deriving a general formula due to its complexity. However this  formula could still be faster than our current approach.

There is currently limited support for using the quotient of two $\mathfrak b$ modules as a $\mathfrak b$ module, and we intend to extend this support in the future. In particular, the modules appearing in \cite{LQ} are of this kind. Furthermore we intend to implement the construction of kernels and cokernels of maps of $\mathfrak b$-modules to provide support for a larger class of modules. Furthermore if $\lambda$ is a character, we obtain a $\mathfrak b$-module $\mathbb C_\lambda$ with trivial $\mathfrak n$-action. Tensoring with such modules changes the weight decomposition, and we intend to implement them. Finally we intend to support the usage of any highest weight representations of $\g$ as $\mathfrak b$-modules.

Our algorithm uses mainly the `standard' BGG resolution, however there are also `parabolic BGG resolution', see \cite{humphreys}. Since we were mainly interested by sheaf cohomology on flag varieties, a quick inspection of the Leray-Hirsch spectral sequence gives that the sheaf cohomology of a homogeneous vector bundle on $G/P$ (corresponding to a $\p$-module $E$) can be computed on $G/B$ simply by restricting $E$ to $\b$. Hence the parabolic BGG resolution was not needed. However, it turns out that such resolutions have applications in differential geometry and could have some applications in physics, because they correspond to invariant differential operators. In this framework, the algebraic expression we found for the BGG maps correspond to the differential operators written in local coordinates. We hope to be able to extend eventually the algorithm to the parabolic setting, in order to compute examples related to physics e.g as explained in \cite{penrose} (page $179$ contains the relevant parabolic subgroups). 

\section{Examples}\label{sec:Examples}

We present some examples of new results obtainable by our algorithm. The code for all the computations in this section is available at \url{https://github.com/RikVoorhaar/bgg-cohomology}.

First we checked all the bigraded cohomology groups from \cite{LQ} and \cite{LQ2}. They contain bigraded tables that compute the center of the principal block of the small quantum group (and have interpretation as certain cohomology groups on $G/B$). We checked Demazure's computation (namely that $T_X$ has no higher cohomology, and $H^0(X,T_X) = \g$ in type $A$) for $A_{\leq 5}$. We also confirmed all results from \cite{vilonen} accessible to our algorithm, for example for $\g = \sl_6$, 
\[
H^k(X,\mathfrak b^{\otimes 4}) = \left\{\begin{array}{lr}
    L(\alpha_1+2\alpha_2+3\alpha_3+2\alpha_4+\alpha_5) & \text{if } k=5\\
    0 & \text{if } k\neq 2,3,5
\end{array}\right.
\]

\subsection{Hochschild cohomology of some flag varieties}

We already introduced the Hochschild cohomology of smooth algebraic varieties before, we now present more complex examples computed with our program. We confirm that $H^i(X, \wedge^kT_X)=0$ for $k \geq 0, i>0$ for types $A_{\leq 4}$, $B_{\leq 3}$, $C_{\leq 3}$, and $G_2$ and give a strong evidence that the Demazure theorem might be generalized. 

\subsubsection{Hochschild cohomology of the flag varieties of type $G_2$}

We compute $H^i(X, \wedge^jT_X)$ for $X$ the flag variety of type $G_2$. We use the fact that $H^i(X,\wedge^j T_X) = H^i(BGG_\bullet(\wedge^j \mathfrak u))$. We note that the only non-trivial cohomology lies in degree 0. For the complete flag variety we obtain the following results:
\begin{itemize}
    \item $\mathrm H^{0}(X_{},\wedge^{0}T_X)=\mathbb{C}$
    \item $\mathrm H^{0}(X_{},\wedge^{1}T_X)=L\left( 3\alpha_{1}+ 2\alpha_{2}\right)$
    \item $\mathrm H^{0}(X_{},\wedge^{2}T_X)=L\left( 2\alpha_{1}+\alpha_{2}\right)\oplus L\left( 3\alpha_{1}+ 2\alpha_{2}\right)\oplus L\left( 6\alpha_{1}+ 3\alpha_{2}\right)$
    \item $\mathrm H^{0}(X_{},\wedge^{3}T_X)=L\left( 4\alpha_{1}+ 2\alpha_{2}\right)^{2}\oplus L\left( 5\alpha_{1}+ 3\alpha_{2}\right)\oplus L\left( 6\alpha_{1}+ 3\alpha_{2}\right)\oplus L\left( 6\alpha_{1}+ 4\alpha_{2}\right)\oplus L\left( 8\alpha_{1}+ 4\alpha_{2}\right)$
    \item $\mathrm H^{0}(X_{},\wedge^{4}T_X)=L\left( 5\alpha_{1}+ 3\alpha_{2}\right)\oplus L\left( 6\alpha_{1}+ 3\alpha_{2}\right)\oplus L\left( 6\alpha_{1}+ 4\alpha_{2}\right)\oplus L\left( 7\alpha_{1}+ 4\alpha_{2}\right)^{2}\oplus L\left( 8\alpha_{1}+ 4\alpha_{2}\right)\oplus L\left( 9\alpha_{1}+ 5\alpha_{2}\right)$
    \item $\mathrm H^{0}(X_{},\wedge^{5}T_X)=L\left( 7\alpha_{1}+ 4\alpha_{2}\right)\oplus L\left( 8\alpha_{1}+ 5\alpha_{2}\right)\oplus L\left( 9\alpha_{1}+ 5\alpha_{2}\right)\oplus L\left( 10\alpha_{1}+ 5\alpha_{2}\right)\oplus L\left( 9\alpha_{1}+ 6\alpha_{2}\right)$
    \item $\mathrm H^{0}(X_{},\wedge^{6}T_X)=L\left( 10\alpha_{1}+ 6\alpha_{2}\right)$
\end{itemize}
The respective dimensions here are $1$, $14$, $98$, $454$, $1226$, $1574$ and $729$. Next we compute the Hochschild cohomology of $X_{\alpha_i}$ where $X_{\alpha_i} = G/P_{\alpha_1}$, with $G = G_2$, and $P_{\alpha_i}$ the parabolic corresponding to $\alpha_i$, for $i = 1$ (the short simple root) and $\alpha_2$ (the long simple root). Interestingly, we note that for $X_2$, $\Gamma(X_2, T_{X_2})$ is bigger than $\g_2$ (it also contains the quasi-minuscule representation). For $X_{\alpha_1}$ we obtain the following:

\begin{itemize}
    \item $\mathrm H^{0}(X_{\alpha_1},\wedge^{0}T_X)=\mathbb{C}$
    \item $\mathrm H^{0}(X_{\alpha_1},\wedge^{1}T_X)=L\left( 3\alpha_{1}+ 2\alpha_{2}\right)$
    \item $\mathrm H^{0}(X_{\alpha_1},\wedge^{2}T_X)=L\left( 3\alpha_{1}+ 2\alpha_{2}\right)\oplus L\left( 6\alpha_{1}+ 3\alpha_{2}\right)$
    \item $\mathrm H^{0}(X_{\alpha_1},\wedge^{3}T_X)=L\left( 6\alpha_{1}+ 3\alpha_{2}\right)\oplus L\left( 6\alpha_{1}+ 4\alpha_{2}\right)\oplus L\left( 8\alpha_{1}+ 4\alpha_{2}\right)$
    \item $\mathrm H^{0}(X_{\alpha_1},\wedge^{4}T_X)=L\left( 6\alpha_{1}+ 4\alpha_{2}\right)\oplus L\left( 9\alpha_{1}+ 5\alpha_{2}\right)$
    \item $\mathrm H^{0}(X_{\alpha_1},\wedge^{5}T_X)=L\left( 9\alpha_{1}+ 6\alpha_{2}\right)$
\end{itemize}
The respective dimensions for $X_{\alpha_1}$ are $1$, $14$, $91$, $336$, $525$, and $273$. For $X_{\alpha_2}$ we obtain the following:
\begin{itemize}
    \item $\mathrm H^{0}(X_{\alpha_2},\wedge^{0}T_X)=\mathbb{C}$
    \item $\mathrm H^{0}(X_{\alpha_2},\wedge^{1}T_X)=L\left( 2\alpha_{1}+\alpha_{2}\right)\oplus L\left( 3\alpha_{1}+ 2\alpha_{2}\right)$
    \item $\mathrm H^{0}(X_{\alpha_2},\wedge^{2}T_X)=L\left( 2\alpha_{1}+\alpha_{2}\right)\oplus L\left( 3\alpha_{1}+ 2\alpha_{2}\right)\oplus L\left( 4\alpha_{1}+ 2\alpha_{2}\right)\oplus L\left( 5\alpha_{1}+ 3\alpha_{2}\right)\oplus L\left( 6\alpha_{1}+ 3\alpha_{2}\right)$
    \item $\mathrm H^{0}(X_{\alpha_2},\wedge^{3}T_X)=L\left( 4\alpha_{1}+ 2\alpha_{2}\right)\oplus L\left( 5\alpha_{1}+ 3\alpha_{2}\right)\oplus L\left( 6\alpha_{1}+ 3\alpha_{2}\right)\oplus L\left( 6\alpha_{1}+ 4\alpha_{2}\right)\oplus L\left( 7\alpha_{1}+ 4\alpha_{2}\right)\oplus L\left( 8\alpha_{1}+ 4\alpha_{2}\right)$
    \item $\mathrm H^{0}(X_{\alpha_2},\wedge^{4}T_X)=L\left( 7\alpha_{1}+ 4\alpha_{2}\right)\oplus L\left( 8\alpha_{1}+ 4\alpha_{2}\right)\oplus L\left( 9\alpha_{1}+ 5\alpha_{2}\right)$
    \item $\mathrm H^{0}(X_{\alpha_2},\wedge^{5}T_X)=L\left( 10\alpha_{1}+ 5\alpha_{2}\right)$
\end{itemize}
The respective dimension for $X_{\alpha_2}$ are $1$, $21$, $189$, $616$, $819$, and $378$.

\subsubsection{Vanishing of higher cohomology}

We have performed the computation of the previous section for other types as well. The fact that higher cohomology vanishes for (partial) flag varieties turns out to hold for these types as well:

\begin{proposition}\label{vanishing}
Let $X = G/P$ be a partial flag variety where $G$ is any of $A_{\leq 4}$, $B_{\leq 3}$, $C_{\leq 3}$, or $G_2$. Then for all $i > 0$ and $k \geq 0$ one has $H^i(X, \wedge^kT_X) = 0$.
\end{proposition}

\subsection{Non-normality of some algebraic varieties}

Let $\mathcal A_r(\g) := \{(x_1, \dots, x_r) \in \g^{\oplus r} : f(x_1, \dots, x_r) = 0  \}$ be the generalized null-cone, or $\mathcal A_r$ for short (it is not reduced in general but we take the associated reduced variety). This is a natural generalisation of the nilpotent cone, see the introduction of \cite{vilonen} for more details. It is possible to study singularities of $\mathcal A_r$ by computing certains map in sheaf cohomology:

\begin{proposition}[\cite{vilonen}]
The variety $\mathcal A_r$ is normal if and only if the map $\psi : \text{Sym}(\g^{\oplus r}) \to H^0(X, \text{Sym}(\u^{\oplus r}))$ is surjective, where $X = G/B$.
\end{proposition}

Here $\psi$ is induced by the natural map $\g \to H^0(G/B, T_X)$. Now if $V$ is a finite-dimensional vector space there is a natural decomposition of $Sym^k(V^{\oplus r})$. By construction $\psi$ commutes with this decomposition, in particular for $k = r$, one summand of $\text{Sym}(V^{\oplus r})$ is given by $V^{\otimes r}$. So if $\g^{\otimes r} \to H^0(X, \u^{\otimes r})$ is not surjective, then $\mathcal A_r$ is not normal. 

\begin{proposition}
The natural map $\g^{\otimes r} \to H^0(X, (\g/\b)^{\otimes r})$ is not surjective for $\g = \b_2$, $r=3$ and $\g = \g_2, r = 2$. In particular the varieties $\mathcal A_3(\b_2)$ and $\mathcal A_2(\g_2)$ are not normal. 
\end{proposition}

\begin{proof}
For $G$ of type $B_2$, our algorithm gives $\dim H^0(X, T_X^{\otimes 3}) = 1024$, since $\dim \b_2^{\otimes 3} =7^3$, there is no surjection $\b_2^{\otimes 3} \to H^0(X, T_X^{\otimes 3})$. Similarly if $G$  is of type $G_2$ we obtain $\dim H^0(G_2, T_X^{\otimes 2}) = 202$ and since $\dim \g_2^{\otimes 2}=12^2$, we conclude that $\g_2^{\otimes 2}$ does not surject onto $H^0(G_2, T_X^{\otimes 2})$.
\end{proof}

\subsection{Some explicit BGG maps}\label{sub:penroseexample}

Now as explained at the end of section $5$, even the explicit form of the BGG maps are interesting on their own.  Recall that for each weight $\lambda \in P$ one has the corresponding line bundle $\mathscr L_{\lambda} := G \times_B \C_{-\lambda}$.  

\begin{definition}
A differential operators $D : \mathscr L_{\lambda} \to \mathscr L_{\mu}$ is a map of sheaves $ \Gamma(U,\mathscr L_{\lambda}) \to \Gamma(U,\mathscr L_{\mu})$ which is locally of the form $\sum_{\alpha} A_{\alpha} \partial^{\alpha}$, where $A_{\alpha}$ are local sections of $Hom(\mathscr L_{\lambda}, \mathscr L_{\mu})$, $x_i$ are local coordinates, and $\partial^{\alpha} = \frac{\partial^{\alpha_1}}{(\partial x_1)^{\alpha_1}} \dots \frac{\partial^{\alpha_m}}{(\partial x_m)^{\alpha_m}} $
\end{definition}

It turns out that left-invariant differential operators between these line bundles are in bijection with Verma modules homomorphisms:

\begin{theorem}[\cite{penrose}]
There is a bijection between $D(\mathscr L_{\lambda},\mathscr L_{\mu})$ and $Hom_{\g}(M(\mu), M(\lambda))$.
\end{theorem}

Hence, the monomials $\mathcal F(x,w)$ from our algorithm correspond under this bijection to left-invariant differential operators between two equivariant line bundles on $G/B$. It would be intersting to see if the explicit expressions in a PBW basis of the BGG maps can be useful, especially outside of type $A$ where no closed formulas are known. We show these maps for $\g = \b_2$ and that  $\lambda = 2 \alpha_1 + \alpha_2$. Then, the non-trivial maps in the BGG complex are $(1 \to 12, 2 \to 21, 12 \to 121, 21 \to 121)$ where e.g $12 \to 121$ correspond to the map of Verma modules $M(s_2 s_1 \cdot \lambda) \to M(s_1 s_2 s_1 \cdot \lambda)$. The maps are given by : \\

$1 \to 12 : {\,60\,f_{12}^{3} +15\,f_{12}\,f_{1}^{2}\,f_{2}^{2} +60\,f_{12}^{2}\,f_{1}\,f_{2} +\,f_{1}^{3}\,f_{2}^{3} -120\,f_{12}\,f_{1}\,f_{122} -30\,f_{122}\,f_{1}^{2}\,f_{2} }$ \\ 

$2 \to 21 : {\,\,120\,f_{12}\,f_{2}\,f_{122} +360\,f_{122}^{2} -24\,f_{122}\,f_{1}\,f_{2}^{2} -6\,f_{12}\,f_{1}\,f_{2}^{3} +12\,f_{12}^{2}\,f_{2}^{2} +\,f_{1}^{2}\,f_{2}^{4}}$ \\

$12 \to 121 : {\,-420\,f_{12}\,f_{2}\,f_{122} +840\,f_{122}^{2} -84\,f_{122}\,f_{1}\,f_{2}^{2} +14\,f_{12}\,f_{1}\,f_{2}^{3} +42\,f_{12}^{2}\,f_{2}^{2} +\,f_{1}^{2}\,f_{2}^{4}}$ \\ 

$21 \to 121 : {-24\,f_{12}^{3} -6\,f_{12}\,f_{1}^{2}\,f_{2}^{2} +18\,f_{12}^{2}\,f_{1}\,f_{2} +\,f_{1}^{3}\,f_{2}^{3} -36\,f_{12}\,f_{1}\,f_{122} +12\,f_{122}\,f_{1}^{2}\,f_{2}}$

\subsection{Dimension of $H^{0}(X, \wedge^kT_X)$}

We computed the dimension of $H^0(X, \wedge^kT_X)$ for $X = G/P$ for all the partial flag varieties described in section \ref{vanishing}. It is interesting to notice that for type other than $A$, we find several ``exotic vector fields'', that is, vector fields which do not come from the map $\g \to H^0(G/P, G \times_P \u)$. Due to  the $\mathbb Z_2$ symmetry in the Dynkin diagram of $A_n$, some of the partial flag varieties are isomorphic, e.g $X_{\alpha_1}\cong X_{\alpha_n}$. In such cases we only list one of the two. In all tables the left column lists the generators of the parabolic subalgebra, and the second row corresponds to the complete flag variety. All the cohomology is concentrated in degree 0.

\begin{table}[!htb]
    \centering
    \begin{tabular}{r|c c c c c}
         k & 0 & 1 & 2 & 3 & 4 \\
         \hline \hline
         &1 & 10 & 50 & 114 & 81\\
         $\alpha_1$ &1 & 15 & 45 & 35 \\
         $\alpha_2$ &1 & 10 & 35 & 30
    \end{tabular}
    \caption{$\dim H^0(X_\alpha, \wedge^k T_X)$ for type $B_2=C_2$}
\end{table}

\begin{table}[!htb]
    \centering
    \begin{tabular}{r|c c c c c c c}
         k & 0 & 1 & 2 & 3 & 4 & 5 & 6\\
         \hline \hline
         &1 & 15 & 105 & 474 & 1225 & 1547 & 729\\
         $\alpha_1$ &1 & 15 & 105 & 359 & 536 & 280 \\
         $\alpha_2$ &1 & 15 & 125 & 419 & 596 & 300\\
         $\alpha_1,\alpha_2$ &1 & 15 & 45 & 35\\
         $\alpha_1,\alpha_3$ &1 & 15 & 90 & 175 & 105
    \end{tabular}
    \caption{$\dim H^0(X_\alpha, \wedge^k T_X)$ for type $A_3$}
\end{table}

\begin{table}[!htb]
    \centering
    \begin{tabular}{r|c c c c c c c c c c}
         k & 0 & 1 & 2 & 3 & 4 & 5 & 6 & 7 & 8 & 9 \\
         \hline \hline
         &1 & 21 & 210 & 1371 & 6839 & 25012 & 59814 & 85009 & 64184 & 19683\\
         $\alpha_1$ &1 & 21 & 217 & 1546 & 7085 & 19557 & 30653 & 24816 & 8008\\
         $\alpha_2$ &1 & 21 & 252 & 2162 & 10480 & 28013 & 41286 & 31424 & 9625\\
         $\alpha_3$ &1 & 21 & 210 & 1329 & 5979 & 17079 & 27734 & 23031 & 7560\\
         $\alpha_1,\alpha_2$ &1 & 28 & 350 & 1680 & 3675 & 3696 & 1386\\
         $\alpha_1,\alpha_3$ &1 & 21 & 210 & 1344 & 4900 & 9302 & 8547 & 3003\\
         $\alpha_2,\alpha_3$ &1 & 21 & 189 & 616 & 819 & 378
    \end{tabular}
    \caption{$\dim H^0(X_\alpha, \wedge^k T_X)$ for type $B_3$}
\end{table}

\begin{table}[!htb]
    \centering
    \begin{tabular}{r|c c c c c c c c c c}
         k & 0 & 1 & 2 & 3 & 4 & 5 & 6 & 7 & 8 & 9 \\
         \hline \hline
         &1 & 21 & 210 & 1413 & 7021 & 25208 & 59730 & 84771 & 64086 & 19683 \\
         $\alpha_1$ &1 & 21 & 210 & 1399 & 6336 & 17856 & 28637 & 23584 & 7700 \\
         $\alpha_2$ &1 & 21 & 224 & 1574 & 7316 & 20376 & 31857 & 25593 & 8190\\
         $\alpha_3$ &1 & 21 & 294 & 2281 & 10179 & 26613 & 39480 & 30465 & 9450\\
         $\alpha_1,\alpha_2$ &1 & 21 & 189 & 3910 & 2205 & 2457 & 1001 \\
         $\alpha_1,\alpha_3$ &1 & 21 & 280 & 1897 & 6643 & 11934 & 10444 & 3528\\
         $\alpha_2,\alpha_3$ &1 & 35 & 280 & 840 & 1050 & 462 
    \end{tabular}
    \caption{$\dim H^0(X_\alpha, \wedge^k T_X)$ for type $C_3$}
\end{table}

\begin{table}[!htb]
    \centering
    \begin{tabular}{r|c c c c c c c c c c c}
         k & 0 & 1 & 2 & 3 & 4 & 5 & 6 & 7 & 8 & 9 & 10\\
         \hline \hline
        & 1 & 24 & 276 & 2023 & 11027 & 45576 & 134773 & 264427 & 319222 & 212178 & 59049\\
        $\alpha_1$ & 1 & 24 & 276 & 2023 & 10403 & 36648 & 82252 & 109723 & 78526 & 23100\\
        $\alpha_2$ & 1 & 24 & 276 & 2273 & 12703 & 45148 & 98552 & 126873 & 87926 & 25200 \\
        $\alpha_1,\alpha_2$ & 1 & 24 & 276 & 1649 & 5476 & 9875 & 8925 & 3150\\
        $\alpha_1,\alpha_3$ & 1 & 24 & 276 & 2174 & 10326 & 27675 & 41000 & 31325 & 9625\\
        $\alpha_1,\alpha_4$ & 1 & 24 & 276 & 1999 & 9151 & 24575 & 37000 & 28800 & 9000 \\
        $\alpha_2,\alpha_3$ & 1 & 24 & 351 & 2274 & 7426 & 12725 & 10900 & 3675 \\
        $\alpha_1,\alpha_2,\alpha_3$ & 1 & 24 & 126 & 224 & 126 \\
        $\alpha_1,\alpha_2,\alpha_4$ & 1 & 24 & 252 & 1248 & 2877 & 3024 & 1176 
    \end{tabular}
    \caption{$\dim H^0(X_\alpha, \wedge^k T_X)$ for type $A_4$}
\end{table}

\newpage{\ }\newpage

\printbibliography 

@article{bott,
    author = "R. Bott",
    title = "Homogenous vector bundles",
    journal = "The Annals of Mathematics",
    series = "2nd Ser.",
    year = "1957",
    volume = "66",
    number = "2",
    pages = "203--248",
}

@article{bowman,
    author = {{Bowman}, C. and {Norton}, E. and {Simental}, J.},
        title = "{Characteristic-free bases and BGG resolutions of unitary simple modules for quiver Hecke and Cherednik algebras}",
         year = "2018",
archivePrefix = {arXiv},
       eprint = {1803.08736},
}

@article{BGG,
    author = "I.N. Bernstein and I. M. Gelfand and S. I. Gelfand",
    title = "Differential operators on the base affine space and a study of $\mathfrak g$–modules",
    journal = "Lie Groups and their Representations",
    year = "1975",
    pages = "21--64",
}

@article{GJ,
    author = "O. Gabber and A. Joseph",
    title = "On the Bernstein-Gelfand-Gelfand resolution and the Duflo sum formula",
    journal = "Compositio Mathematica",
    year = "1981",
    volume = "43",
    number = "1",
    pages = "107--131",
}

@book{humphreys,
    author = "J. Humphreys",
    title = "Representation of semisimple Lie algebras in the BGG category $\O$",
    series = "Graduate Studies in Mathematics",
    volume = "94",
    publisher = "AMS",
    year = "2008",
}

@article{LQ,
    author = {{Lachowska}, Anna and {Qi}, You},
        title = "{The center of small quantum groups I: the principal block in type A}",
         year = "2016",
archivePrefix = {arXiv},
       eprint = {1604.07380},
}

@article{LQ2,
    author = {{Lachowska}, Anna and {Qi}, You},
        title = "{The center of small quantum groups II: singular blocks}",
         year = "2017",
archivePrefix = {arXiv},
       eprint = {1703.02457},
}

@article{mm,
    author = {V. Mazorchuk and R. Mrden},
        title = "{BGG complexes in singular blocks of category $\mathcal O$}",
         year = "2019",
archivePrefix = {arXiv},
       eprint = {1907.04121},
}

@book{penrose,
  title={The Penrose Transform: Its Interaction with Representation Theory},
  author={Baston, R.J. and Eastwood, M.G.},
  series={Dover Books on Mathematics},
  year={2016},
}

@article{rocha,
    author = "A. Rocha-Caridi",
    title = "Splitting criteria for $\g$-modules induced from a parabolic and the Bernstein-Gelfand-Gelfand resolution of a finite-dimensional irreducible $\g$-module",
    journal = " Transactions of the American Mathematical Society",
    year = "1980",
    volume = "262",
    number = "2",
    pages = "335--366",
}

@article{schechtman,
    author = "M. Falk and V. Schechtman and A. Varchenko ",
    title = "BGG complex via configurations spaces",
    journal = "Journal de l'École polytechnique - Mathématiques",
    year = "2014",
    volume = "1",
    pages = "225--245",
}

@misc{flint,
    shorthand = {FLINT},
    author = {W. Hart and F. Johansson and S. Pancratz},
    title = {{FLINT}: {F}ast {L}ibrary for {N}umber {T}heory},
    note = {\url{http://flintlib.org}}
}

@misc{linbox,
    shorthand = {LinBox},
    author = {The LinBox Group},
    title = {{Project LinBox}: Exact computational linear algebra},
    note = {\url{https://linalg.org}}
}

@manual{sagemath,
  shorthand = {SageMath},
  Key          = {SageMath},
  Author       = {{The Sage Developers}},
  Title        = {{S}ageMath, the {S}age {M}athematics {S}oftware {S}ystem ({V}ersion 8.4)},
  note         = {{\tt https://www.sagemath.org}},
  Year         = {2019},
}

@phdthesis{belmans,
    author = {Pieter Belmans},
    title = {Connections between Commutative and Noncommutative Algebraic Geometry},
    school = {University of Antwerp and University of Hasselt},
    year = 2017
}

@article{vilonen,
    author = {K. Vilonen and T. Xue},
    title = {The null-cone and cohomology of vector bundles on flag manifolds},
    year = 2015,
    archivePrefix = {arXiv},
    eprint = {1505.07619}
}

@article{Xiao,
    author = {W. Xiao},
    title = {Differential equations and singular vectors in Verma modules over $\sl(n,\C)$},
    year = 2015,
    archivePrefix = {arXiv},
    eprint = {1503.06385}
}

@article{Xiao2,
    author = {W. Xiao},
    title = {Differential-operator representations of Weyl group and singular vectors},
    year = 2017,
    archivePrefix = {arXiv},
    eprint = {1703.01098}
}

@article{Lutsyuk,
    author = {A. V. Lutsyuk},
    title = {Homomorphisms of the modules $M_\chi$},
    year = 1974,
    volume = 8,
    issue = 4,
    pages = {91--92},
    journal = {Funktsional. Anal. i Prilozhen.},
}

@article{MFF,
    author={F. G. Malikov and B. L. Feigin and D. B. Fuchs},
    title = {Singular vectors in Verma modules over Kac--Moody algebras},
    journal = {Funktsional. Anal. i Prilozhen.},
    year = 1986,
    volume = 20,
    issue = 2,
    pages = {25--37},
}

\end{document}